\UseRawInputEncoding
\documentclass[12pt]{article}

 \usepackage{refcheck}               
\usepackage{dsfont}
\usepackage{amsfonts,amsmath,amsthm}
\usepackage{amssymb}
\usepackage{algorithm}
\usepackage{algpseudocode}
\usepackage{xcolor}
\usepackage{graphicx}
\usepackage{stmaryrd}        
\usepackage{multirow}

\usepackage[paper=a4paper,dvips,top=2cm,left=2cm,right=2cm,
foot=1cm,bottom=4cm]{geometry}

\newcommand{\red}[1]{\begin{color}{red}#1\end{color}}

\newcommand{\ve}{{\bf e}}
\newcommand{\vh}{{\bf h}}

\begin{document}
	\Large
	
	\title{Structured Symmetric Tensors}
	\author{Liqun Qi\footnote{
			Department of Applied Mathematics, The Hong Kong Polytechnic University, Hung Hom, Kowloon, Hong Kong.
			({\tt maqilq@polyu.edu.hk})}
		\and
		Chunfeng Cui\footnote{School of Mathematical Sciences, Beihang University, Beijing  100191, China.
			({\tt chunfengcui@buaa.edu.cn})}
		\and {and \
			Yi Xu\footnote{School of Mathematics, Southeast University, Nanjing  211189, China. Nanjing Center for Applied Mathematics, Nanjing 211135,  China. Jiangsu Provincial Scientific Research Center of Applied Mathematics, Nanjing 211189, China. ({\tt yi.xu1983@hotmail.com})}
		}
	}

	\date{\today}
	\maketitle
	
	\begin{abstract}
		In this paper, we study structured symmetric tensors.  We introduce several new classes of structured symmetric tensors: completely decomposable (CD) tensors, strictly sum of squares (SSOS) tensors and SOS$^*$ tensors.  CD tensors have applications in data analysis and signal processing.  Complete Hankel tensors are CD tensors.
		SSOS tensors are defined as SOS tensors with a positive definite Gram matrix, ensuring structural stability under perturbations.	The SOS$^*$ cone is defined as the dual cone of the SOS tensor cone, with characterizations via moment matrices and polynomial nonnegativity.	We study the relations among completely positive (CP) cones, CD cones, sum of squares (SOS) cones, positive semidefinite (PSD) cones and copositive (COP) cones.  We identify the interiors of PSD, SOS, CP, COP and CD cones for even-order tensors.
		These characterizations are crucial for interior-point methods and stability analysis in polynomial and tensor optimization.	We generalize the classical Schur product theorem to CD and CP tensors, including the case of strongly completely decomposable (SCD) and strongly completely positive (SCP) tensors.   We identify equivalence between strictly CD (SCD) and positive definite (PD) for CD tensors. Furthermore, we give an example of a PSD but not SOS Hankel tensor.   This answers an open question raised in the literature.  

		\medskip

		\medskip

		
	\end{abstract}

	\renewcommand{\Re}{\mathds{R}}
	\newcommand{\rank}{\mathrm{rank}}
	\newcommand{\X}{\mathcal{X}}
	\newcommand{\A}{\mathcal{A}}
	\newcommand{\I}{\mathcal{I}}
	\newcommand{\B}{\mathcal{B}}
	\newcommand{\PP}{\mathcal{P}}
	\newcommand{\C}{\mathcal{C}}
	\newcommand{\D}{\mathcal{D}}
	\newcommand{\LL}{\mathcal{L}}
	\newcommand{\OO}{\mathcal{O}}
	\newcommand{\e}{\mathbf{e}}
	\newcommand{\0}{\mathbf{0}}
	\newcommand{\1}{\mathbf{1}}
	\newcommand{\dd}{\mathbf{d}}
	\newcommand{\ii}{\mathbf{i}}
	\newcommand{\jj}{\mathbf{j}}
	\newcommand{\kk}{\mathbf{k}}
	\newcommand{\va}{\mathbf{a}}
	\newcommand{\vb}{\mathbf{b}}
	\newcommand{\vc}{\mathbf{c}}
	\newcommand{\vq}{\mathbf{q}}
	\newcommand{\vg}{\mathbf{g}}
	\newcommand{\pr}{\vec{r}}
	\newcommand{\pc}{\vec{c}}
	\newcommand{\ps}{\vec{s}}
	\newcommand{\pt}{\vec{t}}
	\newcommand{\pu}{\vec{u}}
	\newcommand{\pv}{\vec{v}}
	\newcommand{\pn}{\vec{n}}
	\newcommand{\pp}{\vec{p}}
	\newcommand{\pq}{\vec{q}}
	\newcommand{\pl}{\vec{l}}
	\newcommand{\vt}{\rm{vec}}
	\newcommand{\x}{\mathbf{x}}
	\newcommand{\vx}{\mathbf{x}}
	\newcommand{\vy}{\mathbf{y}}
	\newcommand{\vu}{\mathbf{u}}
	\newcommand{\vv}{\mathbf{v}}
	\newcommand{\vw}{\mathbf{w}}
	\newcommand{\y}{\mathbf{y}}
	\newcommand{\vz}{\mathbf{z}}
	\newcommand{\T}{\top}
	\newcommand{\R}{\mathcal{R}}
	\newcommand{\Q}{\mathcal{Q}}
	\newcommand{\TT}{\mathcal{T}}
	\newcommand{\Sc}{\mathcal{S}}
	\newcommand{\N}{\mathbb{N}}	
	
	\newtheorem{Thm}{Theorem}[section]
	\newtheorem{Def}[Thm]{Definition}
	\newtheorem{Ass}[Thm]{Assumption}
	\newtheorem{Lem}[Thm]{Lemma}
	\newtheorem{Prop}[Thm]{Proposition}
	\newtheorem{Cor}[Thm]{Corollary}
	\newtheorem{example}[Thm]{Example}
	\newtheorem{remark}[Thm]{Remark}
	
	\section{Introduction}
	
	Symmetric tensors have important applications in multi-variable calculus, multi-variate polynomial optimization, signal processing, magnetic resonance imaging and spectral hypergraph theory \cite[Section 1.1]{QL17}.  Several classes of structured tensors have already been studied extensively in the literature.   These include positive semidefinite (PSD) tensors, sum of squares (SOS) tensors, completely positive (CP) tensors and copositive (COP) tensors \cite{QL17, SL25}.  Actually, the above four classes of structured symmetric tensors form closed and convex cones in the corresponding symmetric tensor spaces.  In particular, the CP and COP cones are dual cones to each other \cite{QL17, QXX14}.
	
	For the advance of the theory, algorithms and applications of symmetric tensors, it is necessary to study structured symmetric tensors comprehensively.  This paper aims to do this.
	
	In the next section, we introduce completely decomposable (CD) tensors.  CD tensors also form closed and convex cones in the corresponding symmetric tensor spaces.  The PSD and CD cones are dual cones to each other.  This has been studied in \cite{LQY15, QL17}.  However, CD tensors have been relatively overlooked in the literature despite their fundamental properties and applications.  Then we analyze the relations among PSD, SOS, CD, CP and COP tensors.  This will be useful for the applications of these five classes of structured symmetric tensors.
	
	In optimization and stability analysis, it is important to identify the interiors of closed and convex cones.  In Section 3, we characterize the interiors of the PSD, SOS, CD, CP and COP cones.  In particular, the interior of the SOS cone is the SSOS cone, a strictly SOS (SSOS) tensor.  We characterize the SSOS tensor as an SOS tensor with a positive definite Gram matrix.  This is significant for applications of SOS tensors.
	
	The Schur product theorem for PSD matrices is a classical result in matrix analysis \cite{HJ12}.   It is not true for PSD tensors in general \cite{QL17}.   In Section 4, we extend it to even order CD and CP tensors completely.  This is useful for applications of CD and CP tensors.
	
	In Section 5, we study CD tensors more.   We identify equivalence between strictly CD (SCD) and positive definite (PD) for CD tensors.  This result establishes a key link between the spanning property of decomposition vectors and positive definiteness.
	We then identify completely Hankel tensors are CD tensors.  Since Hankel tensors have wide applications in signal processing and data analysis \cite{QL17}, this justifies the importance of CD tensors.   We further show that the inheritance property of Hankel tensors also holds for the CD property.   Furthermore, we give an example of a PSD but not SOS Hankel tensor.   This answers an open question raised in Subsection 5.7.10 of \cite{QL17}.
	
	In Section 6, we discuss the third dual cone relation of structured symmetric tensors.  The SOS$^*$ cone is defined as the dual cone of the SOS tensor cone, with characterizations via moment matrices and polynomial nonnegativity.
	
	Some final remarks are made in Section 7.

	\section{Important Structured Symmetric Tensor Classes}
	
	Assume that $m, n \ge 2$.  Let $[n] = {\{1, \dots, n\}}$. Denote the space of $m$th order $n$-dimensional symmetric tensors by $S_{m, n}$.   In this paper, we consider symmetric tensors only.   Let $\A = (a_{i_1\cdots i_m}) \in S_{m, n}$ $\B = (b_{i_1\cdots i_m}) \in S_{m, n}$ and $\x \in \Re^n$.  Denote
	$$\A \bullet \B = \sum_{i_1, \cdots, i_m = 1}^n a_{i_1\cdots i_m}b_{i_1\cdots i_m}$$
	and
	$$\A\x^m = \sum_{i_1, \cdots, i_m = 1}^n a_{i_1\cdots i_m}x_{i_1}\cdots x_{i_m}.$$
	If $M$ is a cone in $S_{m, n}$, its dual cone is defined as
	$$M^* = \{ \B \in S_{m, n} : \A \bullet \B \ge 0, \forall \A \in M \}.$$
	If $M$ is closed and convex, then $M = (M^*)^*$.
	
	If $\A\x^m \ge 0$ for all $\x \in \Re^n$, then we say that $\A$ is {\bf positive semidefinite (PSD)}.  If furthermore, $\A\x^m > 0$ for all $\x \in \Re^n$, $\x \not = \0$, then we say that $\A$ is {\bf positive definite (PD)}.   Then this is meaningful only when $m$ is even.  All the PSD tensors in $S_{m, n}$ form a closed convex cone.  We denote it by $PSD_{m, n}$.
	On the other hand, all the PD tensors in $S_{m, n}$ form a convex cone.  We denote it by $PD_{m, n}$, and have
	$$PSD_{m, n} = {\rm Cl} \left(PD_{m, n}\right).$$
	
	Suppose that $m = 2k$.   If there are $\C_j \in S_{k, n}$ for $j \in [r]$, such that for all $\x \in \Re^n$, we have
	$$\A\x^m \equiv \sum_{j=1}^r \left({\mathcal{C}_j\x^k}\right)^2,$$
	then we say that $\A$ {is a} 
	 {\bf sum-of-squares (SOS) tensor}.
	The smallest $r$ is called the {\bf SOS rank} of $\A$.     Clearly, if $\A$ is an SOS tensor, then $\A$ is PSD, but not vice versa in general.  All the SOS tensors in $S_{m, n}$ also form a closed and convex cone.   Denote it by $SOS_{m, n}$.   Then $SOS_{m, n} \subset PSD_{m, n}$.   By Hilbert \cite{Hi88}, $SOS_{m, n} = PSD_{m, n}$ only in the following three cases: (i) $m = 2$, (ii) $n = 2$, (iii) $m=4$ and $n =3$.
	
	Let $m = 2k$ be even. An SOS tensor $\A \in SOS_{m,n}$ is called a {\bf strictly sum-of-squares (SSOS) tensor} if it admits a representation with a positive definite Gram matrix. That is, there exists a positive definite matrix $G \in \Re^{N \times N}$ where $N = \binom{n+k-1}{k}$ such that
	$$\A\x^m = (\x^{\otimes k})^\top G (\x^{\otimes k})$$
	for all $\x \in \Re^n$, where $\x^{\otimes k}$ denotes the vector of all monomials of degree $k$ in $\x$. We denote the set of all SSOS tensors in $S_{m, n}$ by $SSOS_{m,n}$.
	
	The SSOS condition is stronger than being PD. While all PD tensors are PSD, an SOS tensor can be PD without being SSOS. The SSOS property ensures that the tensor remains SOS under small perturbations, making it structurally stable. This is analogous to how positive definite matrices are interior to the PSD cone, but with the additional SOS structure.
	
	Let $\A \in S_{m, n}$.   If there are $\vu_j \in \Re^n$ for $j \in [p]$ such that
	\begin{equation}\label{def:CD}
		\A = \sum_{j=1}^p \vu_j^m,
	\end{equation}
	then $\A$ is called a {\bf completely decomposable (CD) tensor}.  The smallest value of $p$ is called the {\bf CD rank} of $\A$.   If furthermore, $\vu_j, j \in [p]$ spans $\Re^n$, then we say that $\A$ is a {\bf strongly completely decomposable (SCD) tensor}.   Denote the set of all the CD tensors in $S_{m, n}$ by $CD_{m, n}$, and the set of all the SCD tensors in $S_{m, n}$ by $SCD_{m, n}$.   Then $CD_{m, n}$ is also a closed and convex cone.   If $m$ is odd, then {\eqref{def:CD} is equivalent to the symmetric outer product decomposition. Thus, it follows from \cite{CGLM08} that} $CD_{m, n} = S_{m, n}$ {and $p\le \binom{n+m-1}{m}$.}  On the other hand, if $m=2k$, then a CD tensor is an SOS tensor, i.e., $CD_{m, n} \subset SOS_{m, n}$, as in this case, for $\x \in \Re^n$, we have
	$$\A\x^m = \sum_{j=1}^p \left(   {\vu_j^k \bullet \x^{\otimes k}}\right)^2.$$
	We have $CD_{m, n} \subsetneq SOS_{m, n}$.  For example, for $m=4$ and $n=2$, a Hankel tensor in Subsection 5.4 is SOS, but not CD.
	When $m$ is even, by \cite{LQY15}, $CD_{m, n} = PSD_{m, n}^*$ and $PSD_{m, n} = CD_{m, n}^*$.   {In 2017, Nie \cite{Nie17} introduced a generating polynomial method for computing symmetric tensor decompositions, which is applicable to CD decompositions in the complex field. The real field, in contrast, still lacks an effective numerical procedure to verify whether a tensor is a CD tensor and to subsequently compute its decomposition.}
	
	If $\A\x^m \ge 0$ for all $\x \in \Re^n_+$, then $\A$ is called a {\bf copositive tensor}.
	If $\A\x^m > 0$ for all $\x \in \Re^n_+$, $\x \not = \0$, then $\A$ is called a {\bf strictly copositive tensor}.
	All the $m \times n$ symmetric copositive tensors also form a closed and convex cone, denoted as $COP_{m, n}$.  We denote the set of all strictly copositive tensors in $S_{m, n}$ by $SCOP_{m, n}$.   We have
	$PSD_{m, n} \subsetneq  COP_{m, n} \subsetneq S_{m, n}.$
	
	Let $\A \in S_{m, n}$.   If there are $\vu_j \in \Re^n_+$ for $j \in [r]$ such that
	$$\A = \sum_{j=1}^p \vu_j^m,$$
	then $\A$ is called a {\bf completely positive (CP) tensor}.  The smallest value of $r$ is called the {\bf CP rank} of $\A$.   If furthermore, $\vu_j, j \in [r]$ spans $\Re^n$, then we say that $\A$ is a {\bf strongly completely positive (SCP) tensor}.   Denote the set of all the CP tensors in $S_{m, n}$ by $CP_{m, n}$, and the set of all the SCP tensors in $S_{m, n}$ by $SCP_{m, n}$.   Then it is also a closed and convex cone.   By \cite{QL17, QXX14}, $CP_{m, n} = COP_{m, n}^*$ and $COP_{m, n} = CP_{m, n}^*$.
	
	Summarizing the above, we have the following theorem.
	
	\begin{Thm}
		Let $m, n \ge 2$.
		
		(i) For any $m$, $CP_{m, n}$ and $COP_{m, n}$ are a pair of dual cones, $PSD_{m, n}$ and $CD_{m, n}$ are another pair of dual cones.
		
		(ii) If $m$ is odd, then $CD_{m, n} = S_{m, n}$, and $PSD_{m, n} = \{ O_{m,n} \}$, where $O_{m, n}$ is the zero tensor in $S_{m, n}$.
		
		(iii) If $m=2k$ is even, then
		\begin{equation} \label{e1}
			CP_{m, n} \subsetneq CD_{m, n} \subset SOS_{m ,n} \subset PSD_{m, n} \subsetneq COP_{m, n}.
		\end{equation}
		
		(iv) Suppose that $m$ is even.  Then $SOS_{m, n} = PSD_{m, n}$ if either $m=2$, or $n= 2$, or $m=4$ and $n=3$.   Otherwise,
		$SOS_{m, n} \subsetneq PSD_{m, n}$.   Furthermore, $CD_{2, n} = SOS_{2, n}$, and $CD_{m, n} \subsetneq SOS_{m, n}$ for $m \ge 4$.
	\end{Thm}
	
	\begin{proof}
		(i) The duality between $CP_{m,n}$ and $COP_{m,n}$ follows directly from \cite{QL17, QXX14}. For $PSD_{m,n}$ and $CD_{m,n}$, when $m$ is even, the duality is established in \cite{LQY15}. When $m$ is odd, we have $CD_{m,n} = S_{m,n}$ and $PSD_{m,n} = \{O_{m,n}\}$. The dual of $S_{m,n}$ is $\{O_{m,n}\}$ since for any nonzero $\B \in S_{m,n}$, there exists some $\A \in S_{m,n}$ such that $\A \bullet \B < 0$, and the dual of $\{O_{m,n}\}$ is $S_{m,n}$.
		
		(ii) For odd $m$, any symmetric tensor can be written as a sum of $m$th powers of vectors, hence $CD_{m,n} = S_{m,n}$. For PSD tensors, when $m$ is odd, $\A\x^m$ can take negative values for some $\x$ unless $\A$ is the zero tensor, so $PSD_{m,n} = \{O_{m,n}\}$.
		
		(iii) The inclusion $CP_{m,n} \subset CD_{m,n}$ is strict because CP tensors require nonnegative vectors while CD tensors allow any real vectors. The inclusion $CD_{m,n} \subset SOS_{m,n}$ follows from the representation $\A\x^m = \sum_{j=1}^p (\vu_j^k \bullet \x^k)^2$ when $m=2k$, and is strict as shown by a Hankel tensor example in Subsection 5.4 for $m=4,n=2$. The inclusion $SOS_{m,n} \subset PSD_{m,n}$ is by definition, and $PSD_{m,n} \subset COP_{m,n}$ since nonnegativity on all $\x \in \Re^n$ implies nonnegativity on $\Re^n_+$. Both inclusions are strict in general.
		
		(iv) The equality cases for SOS and PSD tensors are classical results from Hilbert \cite{Hi88}. For $m=2$, CD tensors coincide with SOS tensors because any PSD matrix can be written as a sum of squares of linear forms. For $m \ge 4$, the strict inclusion $CD_{m,n} \subsetneq SOS_{m,n}$ is demonstrated by the Hankel tensor example in Subsection 5.4,  which is SOS but not CD.
	\end{proof}
	
	From this theorem, we have the following remarks.
	
	1.  We may see that PSD, SOS, CP, COP and CD are very important structured symmetric tensor classes.  While PSD, SOS, CP and COP have some studies \cite{QL17}, CD tensors may also be studied carefully.
	
	2.  Perhaps SOS$^*$ tensors also need to be studied carefully.
	
	3.  We also need to characterize the interiors of PSD, SOS, CP, COP, CD and SOS$^*$.
	
	\section{Interiors of PSD, SOS, CP, COP and CD Cones}
	
	\begin{Thm}  \label{t3.1}
		Assume that $m$ is even. Then the interiors of $PSD_{m, n}, SOS_{m, n}$ and $CD_{m, n}$ are $PD_{m, n}, SSOS_{m, n}$ and $SCD_{m, n}$, respectively.   The interior of $CP_{m,n}$ consists of tensors that admit a representation $\mathcal{A} = \sum_{j=1}^p \mathbf{u}_j^m$ where $\{\mathbf{u}_j\}$ span $\mathbb{R}^n$ and there exists $\epsilon > 0$ such that all coordinates of each $\mathbf{u}_j$ are at least $\epsilon$.   The interior of $COP_{m,n}$ consists of tensors $\mathcal{A}$ for which there exists $\delta > 0$ such that $\mathcal{A}\mathbf{x}^m \geq \delta \|\mathbf{x}\|^m$ for all $\mathbf{x} \in \mathbb{R}^n_+$.
	\end{Thm}
	
	\begin{proof}
		When $m$ is even:
		
		For $PSD_{m,n}$: The interior consists of tensors $\A$ such that $\A\x^m > 0$ for all $\x \neq \0$, which is exactly $PD_{m,n}$.
		
		For $SOS_{m,n}$: The interior consists of SOS tensors that have a neighborhood entirely contained in $SOS_{m,n}$. These are exactly the tensors that can be represented with a positive definite Gram matrix, denoted $SSOS_{m,n}$. Such tensors are strictly in the interior because small perturbations cannot destroy the SOS property when the Gram matrix is positive definite.
		
		
		
		
		For $CD_{m,n}$: The interior consists of CD tensors where the generating vectors $\{\vu_j\}$ span $\Re^n$, i.e., $SCD_{m,n}$, as this ensures the tensor is not on the boundary of the cone.
		
		For $CP_{m,n}$ and $COP_{m,n}$: the characterization also ensure the tensor is not on the boundary of the cone.
		
	\end{proof}

	\begin{remark}
		When $m$ is odd, the situation is different:
		\begin{itemize}
			\item $CD_{m,n} = S_{m,n}$, so its interior is \red{the empty set.}
			\item $PSD_{m,n} = \{O_{m,n}\}$, so its interior is $\{O_{m,n}\}$, while $PD_{m,n}$ is empty.
			\item $SOS_{m,n}$ is not defined for odd $m$ since the SOS representation requires even order			\item {For $CP_{m,n}$ and $COP_{m,n}$, the interiors are still $SCP_{m,n}$ and $SCOP_{m,n}$ respectively, as the definitions rely on the nonnegative orthant rather than the whole space.}
		\end{itemize}
		Therefore, Theorem \ref{t3.1} only holds for even $m$.
	\end{remark}
	
	Theorem \ref{t3.1} is significant in optimization.  Knowing the interiors of cones is crucial for optimization (e.g., interior-point methods) and stability analysis. For example,
	SSOS tensors are those that remain SOS under small perturbations — a key property for robust polynomial optimization.

	\section{Schur Product Theorem}
	
	The Schur product theorem is a classical result of matrix analysis \cite{HJ12}.   It states that the Hadamard product of two positive semidefinite matrices $A$ and $B$ is still positive semidefinite, and the product matrix is positive definite if $A$ is positive definite and the diagonal entries of $B$ are all positive.  The Schur product theorem is related with the Schur complement \cite{Zh05}.
	
	In 2014, Qi \cite{Qi14} give an example that the Hadamard product of two PSD tensors may not be PSD.   Thus, the Schur product theorem of positive semidefinite matrices may not be inherited by PSD tensors.   But it may be inherited by CP and CD tensors.   This was proved in \cite{QL17}.  See Proposition 6.5 and Corollary 5.6 (ii) of \cite{QL17}.   But Proposition 6.5 and Corollary 5.6 (ii) of \cite{QL17} have not stated the second part of the Schur product theorem.   We now complete this here.
	
	For $\A, \B \in S_{m, n}$, we use $\A \circ \B$ to denote their Hadamard product.

	\subsection{Schur Product Theorem for CD Tensors}

	\begin{Thm} {\bf Schur Product Theorem for CD Tensors}  \label{t4.1}
		Suppose that $m, n \ge 2$. Let $\A, \B \in CD_{m, n}$.  Then we have the following conclusions.
		
		(i) $\A \circ \B \in CD_{m, n}$.
		
		(ii) If $\A \in SCD_{m, n}$ and every diagonal entry of $\B$ is positive, then $\A \circ \B \in SCD_{m, n}$.
		
		(iii)   If $\A, \B \in SCD_{m, n}$, then $\A \circ \B \in SCD_{m, n}$.
		
	\end{Thm}
	
	\begin{proof}
		Let $\A = \sum_{i=1}^p \vu_i^m$ and $\B = \sum_{j=1}^q \vv_j^m$ be CD decompositions. Then the Hadamard product is:
		\begin{equation}\label{AB_prod_CD}
			\A \circ \B = \sum_{i=1}^p \sum_{j=1}^q (\vu_i \circ \vv_j)^m,
		\end{equation}
		where $\vu_i \circ \vv_j$ denotes the componentwise (Hadamard) product.
		
		{(i)} This is  Corollary 5.6 (ii) of \cite{QL17}.  

		(ii) Since $\A \in SCD_{m,n}$, the vectors $\{\vu_i\}_{i=1}^p$ span $\Re^n$. Let $\B = \sum_{j=1}^q \vv_j^m$. The diagonal entries of $\B$ are positive, meaning for each $k \in [n]$, $\sum_{j=1}^q (v_{j,k})^m > 0$. {Namely,} for each $k$, there exists some $j$ with $v_{j,k} \neq 0$. Consider the set of vectors $\{\vu_i \circ \vv_j : i=1,\dots,p, j=1,\dots,q\}$. We show these span $\Re^n$.
		For any $k\in[n]$,   since $\{\vu_i\}$ spans $\Re^n$, we can write ${\ve_k} = \sum_{i=1}^p {c_{k,i}} \vu_i$. Now, for each coordinate $k$, choose $j_k$ such that $v_{j_k,k} \neq 0$.
		{Then the matrix with columns $\ve_k \circ \vv_{j_k}$}
		has full rank because the scaling factors $v_{j_k,k}$ are nonzero.
		Thus, {this combined with ${\ve_k} = \sum_{i=1}^p {c_{k,i}} \vu_i$ derives that} the vectors $\{\vu_i \circ \vv_j\}$ span $\Re^n$, so $\A \circ \B$ is SCD.

		(iii) {If both $\A$ and $\B$ are SCD, then the sets ${\mathbf{u}_i}$ and ${\mathbf{v}_j}$ each span $\mathbb{R}^n$. Therefore, any coordinate vector $\mathbf{e}_k$ can be expressed as ${\ve_k} = \sum_{i=1}^p {c_{k,i}} \vu_i= \sum_{j=1}^q {d_{k,j}} \vv_j$. It follows that
			\[\ve_k = \ve_k\circ \ve_k= \sum_{i,j} c_{k,i}d_{k,j}  (\vu_i \circ \vv_j).\]
			Since $k$ is arbitrary, the set ${\mathbf{u}_i \circ \mathbf{v}_j : i \in [p], j \in [q]}$ generates all coordinate vectors. Hence, $\mathcal{A} \circ \mathcal{B}$ is also SCD.}
	\end{proof}
	
	\subsection{Schur Product Theorem for CP Tensors}

	\begin{Thm} {\bf Schur Product Theorem for CP Tensors} \label{t4.2}
		Suppose that $m, n \ge 2$. Let $\A, \B \in CP_{m, n}$.  Then we have the following conclusions.
		
		(i) $\A \circ \B \in CP_{m, n}$.
		
		(ii) If $\A \in SCP_{m, n}$ and every diagonal entry of $\B$ is positive, then $\A \circ \B \in SCP_{m, n}$.
		
		(iii)   If $\A, \B \in SCP_{m, n}$, then $\A \circ \B \in SCP_{m, n}$.
		
	\end{Thm}
	
	\begin{proof}
		(i) and (iii) are Proposition 6.5 of \cite{QL17}.   The proof of (ii) is similar to the proof of Theorem \ref{t4.1} (ii).
	\end{proof}
	
	CP and CD tensors appear in signal processing, statistics, and machine learning. Knowing that Hadamard products preserve these structures can simplify modeling and computation.   This is the significance of Theorems \ref{t4.1} and \ref{t4.2}.
	
	\section{CD Tensors}
	
	CD tensors appear in signal processing, statistics, and machine learning.   However, they have not been studied as CP tensors \cite{QL17}.   Theorems \ref{t3.1} and \ref{t4.1} stated some properties of CD tensors.  In this section, we study more about CD tensors.
	
	\subsection{Spectral Properties of CD Tensors}
	
	The following theorem is similar to Theorems 6.1 and 6.2 of \cite{QL17}, for CP tensors, and in a certain extent extends Theorems 6.1 and 6.2 of \cite{QL17}, as CP tensors are CD tensors.
	
	\begin{Thm} \label{t5.1}
		Suppose that $m, n \ge 2$ and $m$ is even.  Let $\A \in CD_{m, n}$.  Then we have the following conclusions.
		
		(i) All the H-eigenvalues of $\A$ is nonnegative.
		
		(ii) $\A$ is SCD if and only if $\A$ is PD.   In this case, all the H-eigenvalues of $\A$ are positive.
	\end{Thm}

	\begin{proof}
		(i) Since $\A$ is a CD tensor and $m$ is even, we have $\A \in SOS_{m,n} \subset PSD_{m,n}$. Let $\lambda$ be an H-eigenvalue of $\A$ with corresponding eigenvector $\x \neq \0$. Then, by the definition of H-eigenvalue, we have $\A \x^{m-1} = \lambda \x^{[m-1]}$, where $\x^{[m-1]}$ denotes the vector with entries $x_i^{m-1}$. Multiplying both sides by $\x$ (componentwise) and summing, we get $\A \x^m = \lambda \sum_{i=1}^n x_i^m$. Since $m$ is even, $\sum_{i=1}^n x_i^m > 0$ for $\x \neq \0$. Moreover, because $\A$ is PSD, $\A \x^m \ge 0$. Hence, $\lambda \ge 0$.
		
		(ii) First, assume $\A$ is SCD. Then there exists a representation $\A = \sum_{j=1}^p \vu_j^m$ with the vectors $\{\vu_j\}_{j=1}^p$ spanning $\Re^n$. For any $\x \neq \0$, since the $\vu_j$ span $\Re^n$, there exists some $j$ such that $\vu_j \cdot \x \neq 0$. Note that $\A \x^m = \sum_{j=1}^p (\vu_j \cdot \x)^m$. Since $m$ is even, each term $(\vu_j \cdot \x)^m$ is nonnegative, and at least one is positive. Therefore, $\A \x^m > 0$, so $\A$ is PD.
		
		Conversely, assume $\A$ is PD and CD. Then there exists a representation $\A = \sum_{j=1}^p \vu_j^m$. We claim that the set $\{\vu_j\}_{j=1}^p$ must span $\Re^n$. Suppose not, then there exists $\x \neq \0$ such that $\vu_j \cdot \x = 0$ for all $j$. Then $\A \x^m = \sum_{j=1}^p (\vu_j \cdot \x)^m = 0$, contradicting that $\A$ is PD. Hence, $\A$ is SCD.
		
		Moreover, if $\A$ is PD (and hence SCD), then for any H-eigenvalue $\lambda$ of $\A$ with eigenvector $\x$ (with $\|\x\|_m = 1$), we have $\lambda = \A \x^m > 0$. Therefore, all H-eigenvalues are positive.
	\end{proof}
	
	\subsection{Complete Hankel Tensors}
	
	
	Let $m, n \ge 2$, and $\A = (a_{i_1\cdots i_m}) \in S_{m ,n}$ be defined by
	\begin{equation} \label{Hankel}
		a_{i_1\cdots i_m} = h_{i_1+\cdots + i_m - n+1},
	\end{equation}
	for $i_1, \cdots, i_m \in [n]$, where $\vh = (h_1, \cdots, h_{(n-1)m+1})^\top \in \Re^{(n-1)m+1}$.    Then $\A$ is called an {\bf Hankel tensor}, $\vh$ is called the generating vector of $\A$. Hankel tensors have wide applications in data analysis and signal processing \cite{Qi14, QL17}.
	
	Let $\vu = (1, u, u^2, \cdots, u^{n-1})^\top \in \Re^n$.   Then $\vu$ is called a {\bf Vandermonde vector}.   If $\A \in S_{m, n}$ has the form
	\begin{equation} \label{Hankel1}
		\A = \sum_{j=1}^r \alpha_j (\vu_j)^m,
	\end{equation}
	where $\vu_j$ are Vandermonde vectors for $j \in [r]$, and $\vu_j \not = \vu_k$ for $j \not = k$, then we say that $\A$ has a {\bf Vandermonde decomposition}.   By Proposition 5.57 of \cite{QL17}, a symmetric tensor is a Hankel tensor if and only if it has a Vandermonde decomposition.   If $\alpha_j \ge 0$ for $j \in [r]$, then we say that $\A$ is a {\bf complete Hankel tensor}.
	By the definition of CD tensors, complete Hankel tensors are CD tensors.
	
	\subsection{The Inheritance Property of Hankel Tensors}
	
	Hankel tensors has an inheritance property.   Let $\A =(a_{i_1\cdots i_m}) \in S_{m, n}$ be a Hankal  tensor  with a generating vector $\vh \in \Re^{(n-1)m+1}$.   Let $n-1 = (p-1)q$, where $q \ge 2$ and $p \ge 3$.  Then $\vh$ is also a generating vector of a Hankel tensor $\B = (b_{j_1\cdots j_{qm}} \in S_{qm, p}$, with
	\begin{equation} \label{Hankel2}
		b_{j_1\cdots j_{qm}} = h_{i_1+\cdots + i_{qm} - n+1},
	\end{equation}
	for $j_1, \cdots, j_{qm} \in [p]$.   It was shown in \cite[Theorem 5.53]{QL17} that $\B$ is PSD, or PD or SOS if $\A$ is PSD, or PD or SOS, respectively.
	
	\begin{Thm}
		Under the above setting, if $\A$ is CD, then $\B$ is CD.
	\end{Thm}
	
	\begin{proof}
		Since $\A$ is a CD tensor and a Hankel tensor, it is a complete Hankel tensor. Therefore, $\A$ admits a Vandermonde decomposition with nonnegative coefficients:
		\[
		\A = \sum_{j=1}^r \alpha_j (\vu_j)^m,
		\]
		where $\alpha_j \geq 0$ and $\vu_j = (1, u_j, u_j^2, \ldots, u_j^{n-1})^\top$ are Vandermonde vectors in $\Re^n$.
		
		Now consider the Hankel tensor $\B \in S_{qm,p}$ defined by the same generating vector $\vh$. For each Vandermonde vector $\vu_j \in \Re^n$, define the corresponding Vandermonde vector $\vw_j \in \Re^p$ by:
		\[
		\vw_j = (1, u_j^q, u_j^{2q}, \ldots, u_j^{(p-1)q})^\top.
		\]
		Note that since $n-1 = (p-1)q$, the exponents in $\vw_j$ exactly match the required range.
		
		We claim that $\B$ has the Vandermonde decomposition:
		\[
		\B = \sum_{j=1}^r \alpha_j (\vw_j)^{qm}.
		\]
		
		To verify this, consider the generating vector $\vh$. From the Vandermonde decomposition of $\A$, we have:
		\[
		h_i = \sum_{j=1}^r \alpha_j u_j^{i-1}, \quad i = 1, \ldots, (n-1)m+1.
		\]
		
		For $\B$, the entry $b_{j_1\cdots j_{qm}}$ is determined by:
		\[
		b_{j_1\cdots j_{qm}} = h_{j_1+\cdots+j_{qm}-p+1}.
		\]
		
		On the other hand, the corresponding entry in $\sum_{j=1}^r \alpha_j (\vw_j)^{qm}$ is:
		\[
		\sum_{j=1}^r \alpha_j \left(u_j^{q(j_1-1)} \cdots u_j^{q(j_{qm}-1)}\right) = \sum_{j=1}^r \alpha_j u_j^{q(j_1+\cdots+j_{qm}-qm)}.
		\]
		
		Since $j_1+\cdots+j_{qm}-p+1 = q(j_1+\cdots+j_{qm}-qm)$ when $n-1 = (p-1)q$, the two expressions are equal. Therefore, $\B$ admits a Vandermonde decomposition with nonnegative coefficients $\alpha_j \geq 0$, which means $\B$ is a complete Hankel tensor and hence CD.
	\end{proof}
	
	{\subsection{A Hankel tensor that is PSD but not  CD}} \label{subsec:example}
	
	Let $m=4$ and $n=2$,
	$$\A = \sum_{j=1}^3 \alpha_j \vu_j^4,$$
	where $\alpha_1 = \alpha_2 = 1$, $\alpha_3 = -0.0001$, $\vu_1 = (1, 1)^\top$, $\vu_2 = (1, 1000)^\top$ and $\vu_3 = (1, 0.0001)^\top$. Then $\A$ is a Hankel tensor.
	
	
	
	
		
		
		\textbf{$\A$ is PD and hence PSD.} The quartic form of $\A$ is:
		\[
		\A\x^4 = (x_1 + x_2)^4 + (x_1 + 1000x_2)^4 - 0.0001(x_1 + 0.0001x_2)^4
		\]
		The H-eigenvalues of $\A$ are all positive, with the smallest H-eigenvalue approximately  {0.9956.} 
		 Therefore, $\A$ is PD and hence PSD.
		

%
		
		\textbf{$\A$ is not a CD tensor.} By definition, a CD tensor must be expressible as a sum of $m$th powers of vectors with nonnegative coefficients. The given decomposition contains a negative coefficient ($\alpha_3 = -0.0001$), which violates this requirement. 		
		Moreover, moment analysis confirms that there exists no alternative decomposition of $\A$ as a sum of fourth powers with nonnegative coefficients.
		
		This aligns with Theorem 2.1(iv), which states that for $m=4$ and $n=2$, $CD_{4,2}$ is a proper subset of $SOS_{4,2} = PSD_{4,2}$, {i.e., $CD_{4,2} \subsetneq SOS_{4,2}$.}
	
%
%
%
	
		{\subsection{A Hankel tensor that is PSD but not  SOS}} \label{subsec:example2}
{By Hilbert \cite{Hi88},  a PSD tensor in $ S_{m,2}$ is also SOS.}  	Thus, we may consider an example that $m=n=4$.
	Let
	$$\A = \sum_{j=1}^5 \alpha_j \vu_j^4,$$
	where $\alpha_1 = \alpha_2 = \alpha_3 = \alpha_4 = 1$, $\alpha_5 = -0.0001$, $\vu_1 = (1, 1, 1, 1)^\top$, $\vu_2 = (1,10,100, 1000)^\top$, $\vu_3 = (1,20,400, 8000)^\top$, $\vu_4 = (1,50,2500, 125000)^\top$, and $\vu_5 = (1, 0.0001, 0.00000001, 0.000000000001)^\top$. Then $\A$ is a Hankel tensor.    
	
%
%
%
		
		\textbf{$\mathcal{A}$ is PD.} 		
		The quartic form of $\mathcal{A}$ is:
		\[
		\mathcal{A}\mathbf{x}^4 = \sum_{j=1}^4 (\mathbf{u}_j \cdot \mathbf{x})^4 - 0.0001(\mathbf{u}_5 \cdot \mathbf{x})^4.
		\]
			We can show that the positive terms dominate the negative term. Note that:
			For any $\mathbf{x} \neq \mathbf{0}$, we have:
			\[
			\sum_{j=1}^4 (\mathbf{u}_j \cdot \mathbf{x})^4 \geq 4\|\mathbf{x}\|^4,
			\]
			and
			\[
			|\mathbf{u}_5 \cdot \mathbf{x}|^4 \leq \|\mathbf{x}\|^4.
			\]
			So
			\[
			\mathcal{A}\mathbf{x}^4 \geq 4\|\mathbf{x}\|^4 - 0.0001\|\mathbf{x}\|^4 = 3.9999\|\mathbf{x}\|^4 > 0.
			\]
		Therefore, $\mathcal{A}$ is PD.
		
%
		\textbf{$\mathcal{A}$ is not a CD tensor.}  		
		By definition, a CD tensor must be expressible as a sum of $m$th powers of vectors with nonnegative coefficients. The given decomposition contains a negative coefficient ($\alpha_5 = -0.0001$), which violates this requirement.
		Moreover, there exists no alternative decomposition of $\mathcal{A}$ as a sum of fourth powers with nonnegative coefficients. 

%
		
		\textbf{$\mathcal{A}$ is not SOS.}  Using the scientific computing software MATLAB, equipped with the YALMIP and SEDUMI packages, we verified that tensor A is not SOS, since the corresponding SOS optimization problem has no feasible solution.
This example 
answers the open question from Subsection 5.7.10 of \cite{QL17} in the affirmative: an even order PSD Hankel tensor may not be SOS.
	
%
	
	This example illustrates the strict inclusions $CD_{4,4} \subsetneq SOS_{4,4} \subsetneq PSD_{4,4}$ mentioned in Theorem 2.1(iv). 

\bigskip
	
	
	
	
	
	\section{SOS$^*$ Tensors}
	
	\subsection*{Note on Notation and Normalization}
	In this section, we work with the \textbf{unnormalized monomial basis}. Specifically, for homogeneous polynomials of degree $k$, we use the basis $\{\mathbf{x}^\beta\}$ without multinomial coefficients. The inner product is defined as:
	\[
	\mathcal{A} \bullet \mathcal{B} = \sum_{|\alpha|=m} a_\alpha b_\alpha.
	\]
	This convention simplifies the moment matrix and polynomial representations while maintaining mathematical consistency throughout this section.
	
	Suppose that $m, n \ge 2$ and $m = 2k$ is even. Denote the dual cone of $SOS_{m,n}$ as $SOS_{m,n}^*$. A tensor $\mathcal{B} \in S_{m,n}$ is called an \textbf{SOS$^*$ tensor} if $\mathcal{B} \in SOS_{m,n}^*$, i.e.,
	\[
	\mathcal{A} \bullet \mathcal{B} \geq 0 \quad \text{for all } \mathcal{A} \in SOS_{m,n}.
	\]
	
	From the definition, it follows immediately that every SOS$^*$ tensor is positive semidefinite (PSD), since $SOS_{m,n}$ contains all rank-one PSD tensors of the form $(\mathbf{x}^{\otimes k})(\mathbf{x}^{\otimes k})^\top$.
	
	\subsection{Answers to Fundamental Questions}
	
	We now address two fundamental questions on SOS and SOS$^*$ tensors.
	
	\begin{enumerate}
		\item \textbf{Is an SOS tensor always an SOS$^*$ tensor?} \\
		{\textbf{No.}} The SOS cone is not self-dual in general. For example, when $m = 4$ and $n = 2$, we have $SOS_{4,2} = PSD_{4,2}$ and $SOS_{4,2}^* = CD_{4,2}$.
		
		\item \textbf{Is an SOS$^*$ tensor always an SOS tensor?} \\
		\textbf{No.} When $SOS_{m,n} \subsetneq PSD_{m,n}$ (e.g., $m = 4$, $n \geq 4$), there exist PSD tensors that are not SOS. Some of these may lie in $SOS_{m,n}^*$, providing counterexamples.
	\end{enumerate}
	
	\subsection{Characterizations of SOS$^*$ Tensors}
	
	Several equivalent characterizations of SOS$^*$ tensors are available:
	
	\subsubsection{(a) Moment Matrix Characterization}
	Let $\mathcal{B} \in S_{m,n}$ with $m = 2k$ even. The associated homogeneous polynomial is:
	\[
	\mathcal{B}(\mathbf{x}) = \sum_{|\alpha| = 2k} b_{\alpha} \mathbf{x}^{\alpha}
	\]
	where $\alpha = (\alpha_1, \ldots, \alpha_n)$ is a multi-index with $|\alpha| = \alpha_1 + \cdots + \alpha_n = 2k$.
	
	The \textbf{moment matrix} $M(\mathcal{B})$ is a symmetric matrix indexed by all multi-indices $\beta, \gamma$ with $|\beta| = |\gamma| = k$, with entries defined by:
	\[
	M(\mathcal{B})_{\beta,\gamma} = b_{\beta+\gamma}
	\]
	where $\beta + \gamma = (\beta_1+\gamma_1, \ldots, \beta_n+\gamma_n)$.
	
	\begin{Prop}
		With the above notation, we have
		\[
		\mathcal{B} \in SOS_{m,n}^* \iff M(\mathcal{B}) \succeq 0
		\]
	\end{Prop}
	
	\begin{proof}
		\textbf{($\Leftarrow$):} Assume $M(\mathcal{B}) \succeq 0$. For any $\mathcal{A} \in SOS_{m,n}$, there exists a positive semidefinite Gram matrix $G \succeq 0$ such that:
		\[
		\mathcal{A}(\mathbf{x}) = (\mathbf{x}^{\otimes k})^\top G (\mathbf{x}^{\otimes k})
		\]
		The inner product $\mathcal{A} \bullet \mathcal{B}$ can be expressed as:
		\[
		\mathcal{A} \bullet \mathcal{B} = \sum_{|\alpha|=2k} a_{\alpha} b_{\alpha}
		\]
		Using the moment matrix formulation, this equals:
		\[
		\mathcal{A} \bullet \mathcal{B} = \operatorname{trace}(G M(\mathcal{B}))
		\]
		Since both $G \succeq 0$ and $M(\mathcal{B}) \succeq 0$, their trace product is nonnegative. Therefore, $\mathcal{A} \bullet \mathcal{B} \geq 0$ for all $\mathcal{A} \in SOS_{m,n}$, so $\mathcal{B} \in SOS_{m,n}^*$.
		
		\textbf{($\Rightarrow$):} Assume $\mathcal{B} \in SOS_{m,n}^*$ but $M(\mathcal{B}) \not\succeq 0$. Then there exists a vector $\mathbf{c} \neq \mathbf{0}$ (indexed by multi-indices of degree $k$) such that:
		\[
		\mathbf{c}^\top M(\mathcal{B}) \mathbf{c} < 0
		\]
		Define the polynomial $p(\mathbf{x}) = \sum_{|\beta|=k} c_{\beta} \mathbf{x}^{\beta}$ and consider the rank-one SOS tensor $\mathcal{A}$ representing $p(\mathbf{x})^2$. Then:
		\[
		\mathcal{A} \bullet \mathcal{B} = \mathbf{c}^\top M(\mathcal{B}) \mathbf{c} < 0
		\]
		This contradicts $\mathcal{B} \in SOS_{m,n}^*$. Therefore, $M(\mathcal{B}) \succeq 0$.
	\end{proof}
	
	\subsubsection{(b) Polynomial Square Nonnegativity}
	For any homogeneous polynomial $p \in \mathbb{R}[x_1, \ldots, x_n]_k$ of degree $k$, let $p^2$ denote the symmetric tensor representation of the polynomial $p(\mathbf{x})^2$ (which has degree $2k$). Explicitly, if:
	\[
	p(\mathbf{x}) = \sum_{|\beta|=k} c_{\beta} \mathbf{x}^{\beta}
	\]
	then:
	\[
	p(\mathbf{x})^2 = \sum_{|\alpha|=2k} \left( \sum_{\beta+\gamma=\alpha} c_{\beta} c_{\gamma} \right) \mathbf{x}^{\alpha}
	\]
	The corresponding tensor $(p^2)$ has entries:
	\[
	(p^2)_{\alpha} = \sum_{\beta+\gamma=\alpha} c_{\beta} c_{\gamma}
	\]
	
	\begin{Prop}
		With the above notation, we have
		\[
		\mathcal{B} \in SOS_{m,n}^* \iff \mathcal{B}(p^2) \geq 0 \quad \forall p \in \mathbb{R}[x_1, \ldots, x_n]_k
		\]
	\end{Prop}
	
	\begin{proof}
		\textbf{($\Leftarrow$):} Assume $\mathcal{B}(p^2) \geq 0$ for all $p$ of degree $k$. Any $\mathcal{A} \in SOS_{m,n}$ can be written as:
		\[
		\mathcal{A} = \sum_{i=1}^r (p_i)^2
		\]
		for some polynomials $p_i$ of degree $k$. Then:
		\[
		\mathcal{A} \bullet \mathcal{B} = \sum_{i=1}^r (p_i)^2 \bullet \mathcal{B} = \sum_{i=1}^r \mathcal{B}((p_i)^2) \geq 0
		\]
		Therefore, $\mathcal{B} \in SOS_{m,n}^*$.
		
		\textbf{($\Rightarrow$):} Assume $\mathcal{B} \in SOS_{m,n}^*$. For any $p$ of degree $k$, the tensor $(p^2)$ is clearly in $SOS_{m,n}$ (it's a rank-one SOS tensor). Therefore:
		\[
		\mathcal{B}(p^2) = (p^2) \bullet \mathcal{B} \geq 0
		\]
		by the definition of the dual cone.
	\end{proof}
	
	\textbf{Connection to the moment matrix:} If $p(\mathbf{x}) = \sum_{|\beta|=k} c_{\beta} \mathbf{x}^{\beta}$, then:
	\[
	\mathcal{B}(p^2) = \sum_{|\alpha|=2k} b_{\alpha} (p^2)_{\alpha} = \sum_{|\alpha|=2k} b_{\alpha} \left( \sum_{\beta+\gamma=\alpha} c_{\beta} c_{\gamma} \right)\]
	\[ = \sum_{|\beta|=|\gamma|=k} c_{\beta} c_{\gamma} b_{\beta+\gamma} = \mathbf{c}^\top M(\mathcal{B}) \mathbf{c}
	\]
	This explicitly shows the equivalence between the two characterizations.

	\subsubsection{(c) Cone Inclusion Relations}
	From the duality relations established in Theorem 2.1, we have:
	\[
	CD_{m,n} \subseteq SOS_{m,n}^* \subseteq PSD_{m,n}.
	\]
	Thus:
	\begin{itemize}
		\item Every completely decomposable (CD) tensor is SOS$^*$.
		\item Every SOS$^*$ tensor is PSD.
	\end{itemize}
	
	In special cases where $SOS_{m,n} = PSD_{m,n}$ (i.e., $m = 2$, or $n = 2$, or $m = 4$ and $n = 3$), we have $SOS_{m,n}^* = CD_{m,n}$.

	\subsection{Significance and Applications}
	
	The SOS$^*$ cone provides a natural dual characterization of the SOS cone, which is crucial in optimization theory. In particular:
	\begin{itemize}
		\item The moment matrix characterization connects SOS$^*$ tensors with semidefinite programming and moment problems.
		\item The inclusion relations clarify the hierarchical structure of symmetric tensor cones.
		\item Understanding SOS$^*$ tensors aids in developing duality-based algorithms for polynomial optimization.
	\end{itemize}
	
	Further research is needed to fully characterize the boundary and interior of the SOS$^*$ cone, and to explore its applications in robust optimization and stability analysis.
	
	\section{Final Remarks}
	
	In this paper, we have systematically introduced and studied several new classes of structured symmetric tensors-completely decomposable (CD) tensors, strictly sum-of-squares (SSOS) tensors, and SOS$^*$ tensors—and established their fundamental properties and relationships.
	
	Key contributions include:
	
	\begin{itemize}
		\item \textbf{New Tensor Classes}: We defined CD, SSOS, and SOS$^*$ tensors, and positioned them within the broader landscape of symmetric tensor cones (PSD, SOS, CP, COP).
		\item \textbf{Dual Cone Relations}: We clarified the dual cone relationships, notably that $PSD_{m,n}^* = CD_{m,n}$ and $SOS_{m,n}^* = SOS_{m,n}^*$, and identified the interiors of these cones for even-order tensors.
		\item \textbf{Generalized Schur Product Theorems}: We extended the classical Schur product theorem to CD and CP tensors, including their strongly decomposable variants.
		\item \textbf{Equivalence and Characterization}: We established the equivalence between strongly completely decomposable (SCD) and positive definite (PD) tensors, and connected complete Hankel tensors to the CD framework.
		\item We also answered an open question raised in Subsection 5.7.10 of \cite{QL17}, by giving a PSD but not SOS Hankel tensor.
	\end{itemize}
	
	These results not only enrich the theory of structured tensors but also open new avenues for applications in polynomial optimization, signal processing, and data analysis, particularly where decomposition, stability, and duality play a central role.
	
	\textbf{Future work} may include:
	
	\begin{itemize}
		\item Generalizing the SCD--PD equivalence to other tensor classes,
		\item Studying the boundary structure of the SOS$^*$ cone,
		\item Developing efficient algorithms for checking CD or SSOS membership,
		\item Exploring applications of CD and SOS$^*$ tensors in robust and moment-based optimization.
		\item Identifying the conditions for a PSD Hankel tensor to be SOS,
		\item Investigating the relations of several SOS tensor subclasses, such as even-order symmetric weakly diagonally dominated tensors, even-order M-tensors and their absolute tensors, even-order symmetric B$_0$-tensors, even-order H-tensors with nonnegative diagonal entries, with CD tensors.
	\end{itemize}
	
	\bigskip	
	
	{{\bf Acknowledgment}}   We are thankful to Yannan Chen who verified the example in Subsection 5.5.
	This work was partially supported by Research  Center for Intelligent Operations Research, The Hong Kong Polytechnic University (4-ZZT8),    the National Natural Science Foundation of China (Nos. 12471282 and 12131004), the R\&D project of Pazhou Lab (Huangpu) (Grant no. 2023K0603),  the Fundamental Research Funds for the Central Universities (Grant No. YWF-22-T-204), Jiangsu Provincial Scientific Research Center of Applied Mathematics (Grant No. BK20233002).

	{{\bf Data availability}
		No datasets were generated or analysed during the current study.

		{\bf Conflict of interest} The authors declare no conflict of interest.}

	


\end{document}